\documentclass[11pt]{amsart}
\usepackage[T1]{fontenc}
\usepackage{amssymb}
\usepackage{bm}
\newtheorem{thm}{Theorem}
\newtheorem{lem}{Lemma}[section]
\newtheorem{prop}{Proposition}[section]
\newtheorem{cor}{Corollary}
\newtheorem{conj}{Conjecture}

\numberwithin{equation}{section}

\begin{document}

\title
{Purely exponential Diophantine equations \\
with four terms of consecutive bases: \\contribution to Skolem's conjecture}

\author{Maohua Le}
\address{Maohua Le
\hfill\break\indent Institute of Mathematics, Lingnan Normal University,
\hfill\break\indent Zhanjiang, Guangdong, 524048.
\hfill\break\indent China
}
\email{lemaohua2008@163.com
}
\author{Takafumi Miyazaki}
\address{Takafumi Miyazaki
\hfill\break\indent Gunma University, Division of Pure and Applied Science,
\hfill\break\indent Graduate School of Science and Technology
\hfill\break\indent Tenjin-cho 1-5-1, Kiryu 376-8515.
\hfill\break\indent Japan
}
\email{tmiyazaki@gunma-u.ac.jp}
\thanks{The second-named author is supported by JSPS KAKENHI (No. 24K06642).}
\today
\maketitle

\markboth
{M.-H. Le \& T. Miyazaki}
{purely exponential Diophantine equations with four terms of consecutive bases}

\begin{abstract}
We study purely exponential Diophantine equations with four terms of consecutive bases.
Notably, we prove that all solutions to the equation 
\[
n^x=(n+1)^y+(n+2)^z+(n+3)^w
\]
in positive integers $n,x,y,z$ and $w$ are given by $(n,x,y,z,w)=(2,5,1,1,2)$, $(3,3,2,1,1)$.
Our proof of this result for each $n \ge 4$ provides an explicit modulus $M$ such that the corresponding equation has no solution already modulo $M$.
This contributes to a classical problem posed by T. Skolem in 1930's on a local-global principle on purely exponential Diophantine equations. 
\end{abstract}

\section{Introduction}

Here we consider Diophantine equations in which all unknowns appear in exponents. 
More precisely, we consider the following equation:
\begin{equation}\label{pexp}
c_1\,{a_{11}}^{x_{11}}\cdots{a_{1\hspace{0.02cm}l_1}}^{x_{1\hspace{0.02cm}l_1}}+
c_2\,{a_{21}}^{x_{21}}\cdots{a_{2\hspace{0.02cm}l_2}}^{x_{2\hspace{0.02cm}l_2}}+
\cdots+
c_k\,{a_{k1}}^{x_{k1}}\cdots{a_{k\hspace{0.02cm}l_k}}^{x_{k\hspace{0.02cm}l_k}}
=0
\end{equation}
in unknown positive integers $x_{i j}$ with $i=1,2,\dots,k$ and $j=1,2,\dots,l_k$, where $k, l_1, l_2, \ldots, l_k$ are given integers with $k \ge 3$ and $l_1 \ge 0,\,l_2 \ge 1, \ldots,\,l_k \ge 1$, and each letter using $a$ or $c$ denotes a fixed nonzero integer not equal to $1$ and a fixed nonzero integer, respectively.
Here we regard the first term of the left side of \eqref{pexp} as the constant $c_1$ when $l_1=0$.
Equation \eqref{pexp} is a special case of $S$-unit equation with $k$-term and also it is usually called a purely exponential Diophantine equation, and many research works on it can be found in literature (cf.~\cite{ShTi,EvGy}).
Much of the interest among them lie in determining the solutions to equation \eqref{pexp}, and in estimating the number of solutions in general cases (cf.~\cite{BaBen,MiPi,MiPi2,MiPi3,ScoSt,BH3,DiHo}). 
In all such cases, restricting the unknown exponents through reduction by appropriate moduli and Baker's theory on linear forms in logarithms are often efficient. 
In particular, it is well-known that Baker's theory allows one to obtain effective upper bounds on the size of the solutions in general, provided that the number of terms $k$ is minimal, that is, $k=3$. 
For the case where $k>3$, it is known as an application of Schmidt Subspace Theorem that in most cases the number of solutions to \eqref{pexp} is finite and it can be bounded from above by an explicit constant depending only on $k$ and the number of base numbers $a_{i j}$.
In this context, it should be noted that effectively evaluating the solutions of any given $S$-unit equation with four or more terms is an important unsolved problem in Diophantine number theory. 

Very recently, the second-named author contributed to the above mentioned problem for $k>3$ by showing the following result:

\begin{prop}[\cite{M}]\label{th-3456}
Let $n$ be any positive integer such that $n \equiv 3 \pmod{4}.$ 
Then the equation
\[
n^x + (n+1)^y + (n+2)^z =(n+3)^w
\]
has no solution in positive integers $x,y,z$ and $w,$ except for $n=3,$ where all such solutions are given by $(x,y,z,w)=(3,1,1,2),(3,3,3,3).$
\end{prop}

The proof of this result uses elementary congruence arguments and Baker's method in both complex and $p$-adic cases.

Inspired by the above work, we consider the following four-term equation as its generalization:
\begin{equation} \label{eq-main}
n^x + \delta_1\, (n+1)^y + \delta_2\, (n+2)^z + \delta_3\, (n+3)^w=0
\end{equation}
in positive integers $x,y,z$ and $w$, where $n$ is any given positive integer greater than 1,  and $\delta_i \in \{1,-1\}$ for $i=1,2,3$ such that $(\delta_1,\delta_2,\delta_3) \ne (1,1,1)$. Thus we have
\begin{alignat}{3}
&n^x+(n+1)^y+(n+2)^z=(n+3)^w & \ \ \ & \text{if $(\delta_1,\delta_2,\delta_3)=(1,1,-1),$}
\label{eq-(1)(1)(-1)}\\
&n^x+(n+1)^y+(n+3)^w=(n+2)^z & \ & \text{if $(\delta_1,\delta_2,\delta_3)=(1,-1,1),$}
\label{eq-(1)(-1)(1)}\\
&n^x+(n+1)^y=(n+2)^z+(n+3)^w & \ & \text{if $(\delta_1,\delta_2,\delta_3)=(1,-1,-1),$} 
\label{eq-(1)(-1)(-1)}\\
&n^x+(n+2)^z+(n+3)^w=(n+1)^y & \ & \text{if $(\delta_1,\delta_2,\delta_3)=(-1,1,1),$}
\label{eq-(-1)(1)(1)}\\
&n^x+(n+2)^z=(n+1)^y+(n+3)^w & \ & \text{if $(\delta_1,\delta_2,\delta_3)=(-1,1,-1),$}
\label{eq-(-1)(1)(-1)}\\
&n^x+(n+3)^w=(n+1)^y+(n+2)^z & \ & \text{if $(\delta_1,\delta_2,\delta_3)=(-1,-1,1),$}
\label{eq-(-1)(-1)(1)}\\
&n^x=(n+1)^y+(n+2)^z+(n+3)^w & \ & \text{if $(\delta_1,\delta_2,\delta_3)=(-1,-1,-1).$}
\label{eq-(-1)(-1)(-1)}
\end{alignat}

Note that equations \eqref{eq-(1)(-1)(1)} and \eqref{eq-(-1)(-1)(1)} always have the solutions $(x,y,z,w)=(1,2,2,1)$ and $(x,y,z,w)=(1,1,1,1)$, respectively.
Also, a simple computer search suggests us to expect all other solutions of our equations are described as the following identities:
\begin{alignat*}{3}
&3^3+4+5=6^2, \ 3^3+4^3+5^3=6^3 
& \ \ & \text{for \eqref{eq-(1)(1)(-1)};}\\
&
2^3+3+5=4^2, \ 2^3+3^5+5=4^4, \ 2^5+3^3+5=4^3, \ 2^7+3+5^3=4^4 
& & \text{for \eqref{eq-(1)(-1)(1)};}\\
&
2+3^3=4+5^2, \ 2^3+3^4=4^3+5^2, \ 2^5+3^2=4^2+5^2, \ 3^3+4=5^2+6
& & \text{for \eqref{eq-(1)(-1)(-1)};}\\
&
3+5^2+6^2=4^3 
& & \text{for \eqref{eq-(-1)(1)(1)};}\\
&
2^2+4=3+5, \ 2^4+4^3=3^2+5, \ 2^6+4=3^3+5^3, \ 3^3+5^2=4^2+6^2
& & \text{for \eqref{eq-(-1)(1)(-1)};}\\
&
2^3+5=3^2+4 
& & \text{for \eqref{eq-(-1)(-1)(1)};}\\
&
2^5=3+4+5^2, \ 3^3=4^2+5+6 
& & \text{for \eqref{eq-(-1)(-1)(-1)}.}
\end{alignat*}

As the main result of this paper, we solve one of our equations completely without assuming any restriction on $n$, as follows:

\begin{thm}\label{th-1230}
Equation \eqref{eq-(-1)(-1)(-1)} has no solution, except for $n=2$ or $3,$ where all solutions are given by $(x,y,z,w)=(5,1,1,2)$ for $n=2,$ and by $(x,y,z,w)=(3,2,1,1)$ for $n=3.$
\end{thm}

We can also provide some miscellaneous results on our equations other than \eqref{eq-(1)(-1)(1)}, \eqref{eq-(-1)(-1)(1)} and \eqref{eq-(-1)(-1)(-1)}, as follows:

\begin{thm}\label{th-miscellaneous}
The following hold.
\begin{itemize}
\item[(i)] 
Assume that $n \equiv 4 \pmod{16}$ or $n \equiv 8 \pmod{16}$ or $n \equiv 1 \pmod{8}$ or $n \equiv 2 \pmod{8}.$
Then equation \eqref{eq-(1)(1)(-1)} has no solution.
\item[(ii)] 
Assume that $n \equiv 4 \pmod{8}$ or $n \equiv 1 \pmod{4}$ or $n \equiv 2 \pmod{8}$ or $n \equiv 3 \pmod{8}.$
Then equation \eqref{eq-(1)(-1)(-1)} has no solution, except for $n=2$ or $3,$ where all solutions are given by $(x,y,z,w)=(1,3,1,2),(3,4,3,2),(5,2,2,2)$ for $n=2,$ and by $(x,y,z,w)=(3,1,2,1)$ for $n=3.$
\item[(iii)] 
Assume that $n \equiv 8 \pmod{16}$ or $n \equiv 12 \pmod{16}$ or $n \equiv 5 \pmod{8}$ or $n \equiv 2 \pmod{4}$ or $n \equiv 3 \pmod{8}.$
Then equation \eqref{eq-(-1)(1)(1)} has no solution, except for $n=3,$ where the only solution is given by $(x,y,z,w)=(1,3,2,2).$
\item[(iv)] 
Assume that $n \equiv 4 \pmod{16}$ or $n \equiv 8 \pmod{16}$ or $n \equiv 1 \pmod{8}$ or $n \equiv 2 \pmod{8}$ or $n \equiv 3 \pmod{8}.$ 
Then equation \eqref{eq-(-1)(1)(-1)} has no solution, except for $n=2$ or $3,$ where all solutions are given by $(x,y,z,w)=(2,1,1,1),(4,3,2,1),$\\$(6,1,3,3)$ for $n=2,$ and by $(x,y,z,w)=(3,2,2,2)$ for $n=3.$
\end{itemize}
\end{thm}

Finally, we will introduce an application of (our proof of) Theorem \ref{th-1230} to Skolem's conjecture. 
Skolem had a profound conjecture regarding a local-global principle on purely exponential Diophantine equations (cf.~\cite[pp.\,398--399]{Sch}, \cite{Sk}). 
He asserts (not explicitly) that if a given purely exponential Diophantine equation has no solution, then it has no solution already modulo an appropriately chosen modulus.
Bert\'ok and Hajdu \cite{BH,BH2} proposed a natural formulation of his conjecture and provided many theoretical and numerical results, and also they gave a heuristic but efficient algorithm to produce an appropriate modulus for many given cases in the sense of the mentioned Skolem's problem (cf.~\cite{Ha}). 
When applied to equation \eqref{pexp}, the conjecture can be stated as follows:
\begin{conj}\label{conj-sk-bh}
Suppose that equation \eqref{pexp} has no solution.
Then there exists some integer $M$ with $M \ge 2$ such that the congruence
\[
c_1\,{a_{11}}^{x_{11}}\cdots{a_{1\hspace{0.02cm}l_1}}^{x_{1\hspace{0.02cm}l_1}}+
c_2\,{a_{21}}^{x_{21}}\cdots{a_{2\hspace{0.02cm}l_2}}^{x_{2\hspace{0.02cm}l_2}}+
\cdots+
c_k\,{a_{k1}}^{x_{k1}}\cdots{a_{k\hspace{0.02cm}l_k}}^{x_{k\hspace{0.02cm}l_k}}
\equiv 0 \mod{M}
\]
does not hold for any positive integers $x_{i j}$ with $i=1,2,\dots,k$ and $j=1,2,\dots,l_k.$
\end{conj}

The proof of Theorem \ref{th-1230} includes the following result, which verifies Conjecture \ref{conj-sk-bh} in special cases.

\begin{cor}\label{cor-skolem}
For any positive integer $n$ with $n \ge 4,$ the congruence
\[
n^x\equiv (n+1)^y+(n+2)^z+(n+3)^w\mod{12(n+1)(n+2)}
\]
does not hold for any positive integers $x,y,z$ and $w.$
\end{cor}

The authors think that this is the first nontrivial result which supports Conjecture \ref{conj-sk-bh} for infinitely many cases beyond three terms.

\section{Congruence restrictions}\label{sec-3}

Let $n$ be a positive integer with $n>1$. 
We consider equation \eqref{eq-main}, that is,
\begin{equation} \label{eq-main-2}
n^x + \delta_1\, (n+1)^y + \delta_2\, (n+2)^z + \delta_3\, (n+3)^w=0
\end{equation}
in positive integers $x,y,z$ and $w$.
In this section, we will find many congruence restrictions on the solutions to \eqref{eq-main-2}. 
Below, we let $(x,y,z,w)$ be any solution to \eqref{eq-main-2}.

\begin{lem}\label{lem-mod3}
The following hold.
\begin{itemize}
\item[(i)] 
If $n \equiv 0 \pmod{3},$ then $(-1)^z=-\delta_1 \delta_2.$
\item[(ii)] 
If $n \equiv 1 \pmod{3},$ then $(-1)^y=\delta_1$ and $\delta_3=1.$
\item[(iii)] 
If $n \equiv 2 \pmod{3},$ then $(-1)^x=\delta_2$ and $(-1)^w=\delta_2 \delta_3.$
\end{itemize}
\end{lem}

\begin{proof}
Reducing equation \eqref{eq-main-2} modulo $3$ gives that 
\[
\begin{cases}
\, \delta_1\,1^y+\delta_2\,2^z + \delta_3\,3^w \equiv 0 \pmod{3} & \text{if $n \equiv 0 \pmod{3},$}\\
\,1^x+\delta_1 \,2^y+\delta_2 \,3^z + \delta_3 \,4^w \equiv 0 \pmod{3} & \text{if $n \equiv 1 \pmod{3},$}\\
\,2^x+\delta_1 \,3^y+\delta_2 \,4^z + \delta_3 \,5^w \equiv 0 \pmod{3} & \text{if $n \equiv 2 \pmod{3}.$}
\end{cases}
\]
Thus
\[
\begin{cases}
\, \delta_1 +\delta_2 (-1)^z \equiv 0 \pmod{3} & \text{if $n \equiv 0 \pmod{3},$}\\
\,1+\delta_1 (-1)^y+ \delta_3 \equiv 0 \pmod{3} & \text{if $n \equiv 1 \pmod{3},$}\\
\,(-1)^x + \delta_2 + \delta_3 (-1)^w \equiv 0 \pmod{3} & \text{if $n \equiv 2 \pmod{3}.$}
\end{cases}
\]
This implies the assertion.
\end{proof}

\begin{cor}\label{cor-n1mod3}
If $n \equiv 1 \pmod{3},$ then all equations $\eqref{eq-(1)(1)(-1)}, \eqref{eq-(1)(-1)(-1)}, \eqref{eq-(-1)(1)(-1)}$ and \eqref{eq-(-1)(-1)(-1)} have no solution.
\end{cor}

\begin{lem}\label{lem-mod4}
The following hold.
\begin{itemize}
\item[(i)] 
If $n \equiv 0 \pmod{4},$ then at least one of the following cases holds.
\begin{itemize}
\item[$\bullet$] $(-1)^w=-\delta_1 \delta_3$ and $z>1.$
\item[$\bullet$] $(-1)^w=\delta_1 \delta_3$ and $z=1.$
\end{itemize}
Further, if $\delta_1=\delta_2=\delta_3$ and $n \equiv 0 \pmod{8},$ then $z=2$ and $w$ is odd.
\item[(ii)] 
If $n \equiv 1 \pmod{4},$ then at least one of the following cases holds.
\begin{itemize}
\item[$\bullet$] $(-1)^z=-\delta_2$ and $y>1.$
\item[$\bullet$] $(-1)^z=\delta_2$ and $y=1.$
\end{itemize}
\item[(iii)] 
If $n \equiv 2 \pmod{4},$ then at least one of the following cases holds.
\begin{itemize}
\item[$\bullet$] $(-1)^y=-\delta_1 \delta_3$ and $x>1.$
\item[$\bullet$] $(-1)^y=\delta_1 \delta_3$ and $x=1.$
\end{itemize}
\item[(iv)] 
If $n \equiv 3 \pmod{4},$ then at least one of the following cases holds.
\begin{itemize}
\item[$\bullet$] $(-1)^x=-\delta_2$ and $w>1.$
\item[$\bullet$] $(-1)^x=\delta_2$ and $w=1.$
\end{itemize}
\end{itemize}
\end{lem}

\begin{proof}
The proof proceeds similarly to that of Lemma \ref{lem-mod3}.
Reducing equation \eqref{eq-main-2} modulo $4$ implies that 
\[
\begin{cases}
\,\delta_1\,1^y + \delta_2\,2^z + \delta_3\,3^w \equiv 0 \pmod{4} & \text{if $n \equiv 0 \pmod{4},$}\\
\,1^x + \delta_1\, 2^y + \delta_2\, 3^z + \delta_3\,4^w \equiv 0 \pmod{4} & \text{if $n \equiv 1 \pmod{4},$}\\
\,2^x + \delta_1\, 3^y + \delta_2\, 4^z + \delta_3\, 5^w \equiv 0 \pmod{4} & \text{if $n \equiv 2 \pmod{4},$}\\
\,3^x + \delta_1\, 4^y + \delta_2\, 5^z + \delta_3\, 6^w \equiv 0 \pmod{4} & \text{if $n \equiv 3 \pmod{4}.$}
\end{cases}
\]
Thus
\[
\begin{cases}
\,\delta_1 + \delta_2\,2^z + \delta_3(-1)^w \equiv 0 \pmod{4} & \text{if $n \equiv 0 \pmod{4},$}\\
\,1 + \delta_1\, 2^y + \delta_2 (-1)^z \equiv 0 \pmod{4} & \text{if $n \equiv 1 \pmod{4},$}\\
\,2^x + \delta_1 (-1)^y + \delta_3 \equiv 0 \pmod{4} & \text{if $n \equiv 2 \pmod{4},$}\\
\,(-1)^x + \delta_2 + \delta_3\, 2^w \equiv 0 \pmod{4} & \text{if $n \equiv 3 \pmod{4}.$}
\end{cases}
\]
This implies the assertion, except the second one in (i).
If $\delta_1=\delta_2=\delta_3$ and $n \equiv 0 \pmod{8},$ then reducing equation \eqref{eq-main-2} modulo $8$ implies that 
\[
1+2^z+3^w \equiv 0 \mod{8}.
\]
This implies that $z=2$ and $w$ is odd.
\end{proof}

\begin{lem}\label{lem-modnplus1}
At least one of the following cases holds.
\begin{itemize}
\item[(i)] 
$(-1)^x=-\delta_2,$ and $n=2^r-1$ for some $r \in \mathbb N$ with $r \ge 2.$
\item[(ii)] 
$(-1)^x=\delta_2,$ $w=1,$ and $2(\delta_2+\delta_3) \equiv 0 \pmod{n+1}.$
\item[(iii)] 
$(-1)^x=\delta_2,$ $w>1,$ and $n \in \{f-1,2f-1\}$ for some odd $f \in \mathbb N$ with $f>1$ such that $f \mid (2^{w-1}+\delta_2 \delta_3).$
\end{itemize}
\end{lem}

\begin{proof}
Reducing equation \eqref{eq-main-2} modulo $n+1$ implies that 
\[
(-1)^x + \delta_2 + \delta_3\,2^w \equiv 0 \mod{n+1}.
\]
We distinguish two cases according to whether $(-1)^x=-\delta_2$ or $\delta_2$.

First, assume that $(-1)^x=-\delta_2$.
Then $2^w \equiv 0 \pmod{n+1}$.
Thus $n+1=2^r$ for some $r \in \mathbb N$ with $r \le w$.
Since $n>1$, one has $r>1$. 

Next, assume that $(-1)^x=\delta_2$.
Then 
\[
2\delta_2 + \delta_3\,2^w \equiv 0 \mod{n+1}.
\] 
If $w=1$, then $2(\delta_2+\delta_3) \equiv 0 \pmod{n+1}$.
It remains to consider the case where $w>1$.
Since 
\[
2\delta_3\,(2^{w-1}+\delta_2 \delta_3) \equiv 0 \mod{n+1},
\]
it follows that $n+1=2^{\nu_{2}(n+1)}f$ with $\nu_{2}(n+1) \le 1$ for some odd $f$ such that $f \mid (2^{w-1}+\delta_2 \delta_3).$
\par
To sum up, this completes the proof.
\end{proof}

\begin{lem}\label{lem-modnplus2}
At least one of the following cases holds.
\begin{itemize}
\item[(i)] 
$(-1)^y=-\delta_1\delta_3,$ and $n=2^s-2$ for some $s \in \mathbb N$ with $s \ge 2.$
\item[(ii)] 
$(-1)^y=\delta_1\delta_3, \ x=1,$ and $2-2\delta_3 \equiv 0 \pmod{n+2}.$
\item[(iii)] 
$(-1)^y=\delta_1\delta_3, \ x>1,$ and $n \in \{g-2,2g-2\}$ for some odd $g \in \mathbb N$  such that $g \mid (2^{x-1}+\delta_3 (-1)^x).$
\end{itemize}
\end{lem}

\begin{proof}
The proof proceeds almost similarly to that of Lemma \ref{lem-modnplus1}.
Reducing equation \eqref{eq-main-2} modulo $n+2$ implies that 
\[
(-2)^x + \delta_1\,(-1)^y + \delta_3 \equiv 0 \mod{n+2}.
\]
We distinguish two cases according to whether $(-1)^y=-\delta_1 \delta_3$ or $\delta_1 \delta_3$.

First, assume that $(-1)^y=-\delta_1 \delta_3$.
Then $2^x \equiv 0 \pmod{n+2}$, and this implies that we are in case (i).

Next, assume that $(-1)^y=\delta_1 \delta_3$.
Then
\[
2^{x}+2\delta_3 (-1)^x \equiv 0 \mod{n+2}.
\]
If $x=1$, then $2-2\delta_3 \equiv 0 \pmod{n+2}$.
It remains to consider the case where $x>1$.
Since
\[
2\,(\,2^{x-1} + \delta_3 (-1)^x\,) \equiv 0 \mod{n+2},
\]
it follows that $n+2=2^{\nu_{2}(n+2)}g$ with $\nu_{2}(n+2) \le 1$ for some odd $g$ such that $g \mid (\,2^{x-1} + \delta_3 (-1)^x\,).$
\par
To sum up, this completes the proof.
\end{proof}

\section{Proof of Theorem \ref{th-1230}}\label{sec-2}

Let $n$ be any positive integer with $n>1$. 
We consider equation \eqref{eq-(-1)(-1)(-1)}, that is,
\begin{equation} \label{eq-th1}
n^x=(n+1)^y+(n+2)^z+(n+3)^w
\end{equation}
in positive integers $x,y,z$ and $w$.
In this section, we let $(x,y,z,w)$ be any solution to this equation.

\begin{lem}\label{lem-xge2}
$x \ge 2.$
\end{lem}

This follows immediately as it is clear that $x>y$ in equation \eqref{eq-th1}.
However, towards Corollary \ref{cor-skolem}, we give another proof (for the case where $n \ge 3$).

\begin{proof}[Proof of Lemma $\ref{lem-xge2}$]
Assume that $n \ge 3$.
Suppose on the contrary that $x=1$.
Reducing equation \eqref{eq-th1} modulo $n+2$ implies that
\[
3+(-1)^y \equiv 0 \mod{n+2}.
\]
However, this does not hold as $0<3+(-1)^y<5 \le n+2$. 
\end{proof}

All four lemmas in the previous section are applied for $(\delta_1,\delta_2,\delta_3)=(-1,-1,-1)$, together with Lemma \ref{lem-xge2}, and we obtain the following lemmas:

\begin{lem}\label{th1-lem-mod3}
The following hold.
\begin{itemize}
\item[(i)] 
If $n \equiv 0 \pmod{3},$ then $z$ is odd.
\item[(ii)] 
If $n \equiv 1 \pmod{3},$ then equation \eqref{eq-th1} has no solution.
\item[(iii)] 
If $n \equiv 2 \pmod{3},$ then $x$ is odd and $w$ is even.
\end{itemize}
\end{lem}

\begin{lem}\label{th1-lem-mod4}
The following hold.
\begin{itemize}
\item[(i)] 
If $n \equiv 0 \pmod{4},$ then at least one of the following cases holds.
\begin{itemize}
\item[$\bullet$] $w$ is odd and $z>1.$
\item[$\bullet$] $w$ is even and $z=1.$
\end{itemize}
Further, if $n \equiv 0 \pmod{8},$ then $z=2$ and $w$ is odd.
\item[(ii)] 
If $n \equiv 1 \pmod{4},$ then at least one of the following cases holds.
\begin{itemize}
\item[$\bullet$] $z$ is even and $y>1.$
\item[$\bullet$] $z$ is odd and $y=1.$
\end{itemize}
\item[(iii)] 
If $n \equiv 2 \pmod{4},$ then $x$ is odd.
\item[(iv)] 
If $n \equiv 3 \pmod{4},$ then at least one of the following cases holds.
\begin{itemize}
\item[$\bullet$] $x$ is even and $w>1.$
\item[$\bullet$] $x$ is odd and $w=1.$
\end{itemize}
\end{itemize}
\end{lem}

\begin{lem}\label{th1-lem-modnplus1}
At least one of the following cases holds.
\begin{itemize}
\item[(i)] 
$x$ is even, and $n=2^r-1$ for some $r \in \mathbb N$ with $r \ge 2.$
\item[(ii)] 
$x$ is odd, $w=1,$ and $n=3.$
\item[(iii)] 
$x$ is odd, $w>1,$ and $n \in \{f-1,2f-1\}$ for some odd $f \in \mathbb N$ such that $f \mid (2^{w-1}+1).$
\end{itemize}
\end{lem}

\begin{lem}\label{th1-lem-modnplus2}
At least one of the following cases holds.
\begin{itemize}
\item[(i)] 
$y$ is odd, and $n=2^s-2$ for some $s \in \mathbb N$ with $s \ge 2.$
\item[(ii)] 
$y$ is even, and $n \in \{g-2,2g-2\}$ for some odd $g \in \mathbb N$ such that $g \mid (2^{x-1}-(-1)^x).$
\end{itemize}
\end{lem}

We are now in position to prove a main case of Theorem \ref{th-1230}.

\begin{lem}\label{lem-6543-ngt3}
If $n \ge 4,$ then equation \eqref{eq-th1} has no solution. 
\end{lem}

\begin{proof}
Assume that $n \ge 4$ and let $(x,y,z,w)$ be any solution to equation \eqref{eq-th1}.
We distinguish four cases according to the residue of $n$ modulo 4.
\subsection{Case where $n \equiv 0 \pmod{4}$} 
Lemmas \ref{th1-lem-mod4}\,(i), \ref{th1-lem-modnplus1}\,(iii) and \ref{th1-lem-modnplus2}\,(iii) say that
\begin{align}
&\text{
either \ $2 \mid w$ and $z>1$ \ or \ $2 \nmid w$ and $z=1;$
}
\label{n0mod4-cond1}\\
& \ 2 \nmid x, \ n=f-1, \ f \in \mathbb N, \ 2 \nmid f, \ f \mid (2^{w-1}+1);
\label{n0mod4-cond2}\\
& \ 2 \mid y, \ n =2g-2, \ g \in \mathbb N, \ 2\nmid g, \ g \mid (2^{x-1}-(-1)^x),
\label{n0mod4-cond3}
\end{align}
respectively.
Since $x$ is odd by \eqref{n0mod4-cond2}, it follows from the last divisibility relation in \eqref{n0mod4-cond3} that $-1$ is a quadratic residue modulo $g$ with $g$ odd, so that $g \equiv 1 \pmod{4}$.
Then, since $n=2g-2$ by \eqref{n0mod4-cond3}, one has to find that $n \equiv 0 \pmod{8}$.
Lemma \ref{th1-lem-mod4}\,(i) says that $z=2$ and $w$ is odd.
Thus Lemma \ref{lem-mod3}\,(i,\,iii) tells us that we have to be in the case where $n \equiv 1 \pmod{3}$.
However, this contradicts Corollary \ref{cor-n1mod3}. 
\subsection{Case where $n \equiv 1 \pmod{4}$} 
Lemmas \ref{th1-lem-modnplus1}\,(iii) and \ref{th1-lem-modnplus2}\,(iii) say that
\begin{align}
&2 \nmid x, \ n=2f-1, \ f \in \mathbb N, \ 2 \nmid f, \ f \mid (2^{w-1}+1);
\label{n1mod4-cond2}\\
&2 \mid y, \ n =g-2, \ g \in \mathbb N, \ 2\nmid g, \ g \mid (2^{x-1}-(-1)^x),
\label{n1mod4-cond3}
\end{align}
respectively.
In the same way as that for case where $n \equiv 0 \pmod{4}$, one finds that $g \equiv 1 \pmod{4}$.
Then, since $n=g-2$ by \eqref{n1mod4-cond3}, one has to find that $n \equiv 3 \pmod{4}$, a contradiction.

\subsection{Case where $n \equiv 2 \pmod{4}$} 
Lemmas \ref{th1-lem-modnplus1}\,(iii) and \ref{th1-lem-modnplus2}\,(iii) say that
\begin{align}
&2 \nmid x, \ n=f-1, \ f \in \mathbb N, \ 2 \nmid f, \ f \mid (2^{w-1}+1);
\label{n2mod4-cond2}\\
&2 \nmid y, \ n =2^s-2, \ s \in \mathbb N, \ s \ge 2,
\label{n2mod4-cond3}
\end{align}
respectively.
Since $n \equiv 2 \pmod{4}$, it follows from \eqref{n2mod4-cond2} that $f=n+1 \equiv 3 \pmod{4}$.
Then, by the last divisibility relation in \eqref{n2mod4-cond2}, one finds that $w$ is even, so that $-2$ is a quadratic residue modulo $f$ with $f$ odd, whereby $f \equiv 3 \pmod{8}$.
Therefore, $n=f-1 \equiv 2 \pmod{8}$.
This together with the second item in \eqref{n2mod4-cond3} implies that $s=2$, namely, $n=2$.

\subsection{Case where $n \equiv 3 \pmod{4}$} 
Lemmas \ref{th1-lem-mod3} and \ref{th1-lem-modnplus2}\,(iii) say that
\begin{align}
&\text{
either \ $n \equiv 0 \pmod{3}$ \ or \ $n \equiv -1 \pmod{3};$
}
\label{n3mod4-cond1}\\
& \ 2 \mid y, \ n =g-2, \ g \in \mathbb N, \ 2\nmid g, \ g \mid (2^{x-1}-(-1)^x),
\label{n3mod4-cond3}
\end{align}
respectively.
If $n \ne 3$, then Lemma \ref{th1-lem-modnplus1}\,(i) says that
\begin{align}
2 \mid x, \ n=2^r-1, \ r \in \mathbb N, \ r \ge 2.
\label{n3mod4-cond2}
\end{align}
Observe that 
\[
n=g-2=2^r-1.
\]
Since $n=2^r-1 \not\equiv -1 \pmod{3}$, it follows from \eqref{n3mod4-cond1} that $n \equiv 0 \pmod{3}$, so that $r$ is even.
Also, since $g=2^r+1$, and $x$ is even by \eqref{n3mod4-cond2}, it follows from the last divisibility relation in \eqref{n3mod4-cond3} that 
\begin{equation}\label{n3mod4-cond4}
2^{x-1} \equiv 1 \mod{(2^r+1)}.
\end{equation}
Finally, write 
\[
x-1=r Q + R
\]
for some integers $Q,R$ with $Q \ge 0$ and $0 \le R<r$.
Note that $R$ is odd as $x-1$ is odd and $r$ is even.
Thus, congruence \eqref{n3mod4-cond4} leads to
\[
2^R \equiv (-1)^Q \mod{(2^r+1)}
\]
with $2^R>1$.
This implies that
\[
2^R - (-1)^Q \ge 2^r+1>2^R+1,
\]
which is clearly absurd.

\vspace{0.2cm}
To sum up, we complete the proof of Lemma \ref{lem-6543-ngt3}.
\end{proof}

It remains to show the following lemma:

\begin{lem}\label{lem-6543-n2n3}
The following hold.
\begin{itemize}
\item[$\bullet$] 
The only solution to the equation
\[ 
2^x=3^y+4^z+5^w 
\]
in positive integers $x,y,z$ and $w$ is given by $(x,y,z,w)=(5,1,1,2).$
\item[$\bullet$] 
The only solution to the equation
\[ 
3^x=4^y+5^z+6^w 
\]
in positive integers $x,y,z$ and $w$ is given by $(x,y,z,w)=(3,2,1,1).$
\end{itemize}
\end{lem}

\begin{proof}
The first assertion follows from a result of Alex and Foster \cite[Theorem 3.2]{AF}, and the second one follows from \cite[Theorem 3.11]{AF}.
\end{proof}

The combination of Lemmas \ref{lem-6543-ngt3} and \ref{lem-6543-n2n3} completes the proof of Theorem \ref{th-1230}.

\section{Proof of Theorem \ref{th-miscellaneous}}\label{sec-4}

Since the proof of Theorem \ref{th-miscellaneous} is similar to that of Theorem \ref{th-1230}, we will briefly argue.

\subsection{Case where $(\delta_1,\delta_2,\delta_3)=(1,1,-1)$}
\label{subsec-(1)(1)(-1)}
Let $(x,y,z,w)$ be any solution to equation \eqref{eq-(1)(1)(-1)}.
We proceed similarly to the proof of Theorem \ref{th-1230}.
Clearly, $w>1$.

\vspace{0.2cm}{\it Case where $n \equiv 0 \pmod{4}.$} \ 
Lemmas \ref{lem-modnplus1}\,(iii) and \ref{lem-modnplus2}\,(iii) say that
\[
2 \mid x; \ \ \ n =2g-2, \ g \in \mathbb N, \ 2 \nmid g, \ g \mid (2^{x-1}-(-1)^x),
\]
respectively. 
Since $x$ is even, it follows from the divisibility relation on $g$ that $g \equiv \pm 1 \pmod{8}$.
Thus $n=2g-2 \equiv 0 $ or $12 \pmod{16}$.

\vspace{0.2cm}{\it Case where $n \equiv 1 \pmod{4}.$} \ 
Lemmas \ref{lem-modnplus1}\,(iii) and \ref{lem-modnplus2}\,(iii) say that
\[
2 \mid x; \ \ \ 
n =g-2, \ g \in \mathbb N, \ 2 \nmid g, \ g \mid (2^{x-1}-(-1)^x),
\]
respectively.
In the same way as that in case where $n \equiv 0 \pmod{4}$, one has $g \equiv \pm 1 \pmod{8}$.
On the other hand, $g=n+2 \equiv 3 \pmod{4}$. 
Thus, $g \equiv -1 \pmod{8}$, so that $n=g-2 \equiv 5 \pmod{8}$.

\vspace{0.2cm}{\it Case where $n \equiv 2 \pmod{4}.$} \ 
Since $x$ is even by Lemma \ref{lem-modnplus1}\,(iii), it follows from Lemma \ref{lem-modnplus2}\,(i) that
\[
n=2^s-2, \ s \in \mathbb N, \ s \ge 2.
\]
Thus, either $n=2$ or $n \equiv 6 \pmod{8}$.

\vspace{0.2cm}
To sum up, we are in one of the following cases:
\begin{alignat*}{4}
&n \equiv 0 \pmod{16}; & \ \ &n \equiv 12 \pmod{16}; & \ \ & n \equiv 5 \pmod{8}; \\
&n=2; & &n \equiv 6 \pmod{8}; & \ \ &n \equiv 3 \pmod{4}.
\end{alignat*}
It is known by \cite[Theorem 3.1]{AF} that there is no soltuion to equation \eqref{eq-(1)(1)(-1)} for $n=2$.
Thus, assertion (i) holds.

\subsection{Case where $(\delta_1,\delta_2,\delta_3)=(1,-1,-1)$}
\label{subsec-(1)(-1)(-1)}

We proceed similarly to Subsection \ref{subsec-(1)(1)(-1)}.

\vspace{0.2cm}{\it Case where $n \equiv 0 \pmod{4}.$} \ 
Lemmas \ref{lem-modnplus1}\,(iii) and \ref{lem-modnplus2}\,(iii) say that
\[
2 \nmid x; \ \ \ 
n =2g-2, \ g \mid (2^{x-1}-(-1)^x),
\]
respectively.
Since $x$ is odd, it follows from the divisibility relation on $g$ that $g \equiv 1 \pmod{4}$.
Thus $n=2(g-1) \equiv 0 \pmod{8}$.

\vspace{0.2cm}{\it Case where $n \equiv 1 \pmod{4}.$} \ 
Lemmas \ref{lem-modnplus1}\,(iii) and \ref{lem-modnplus2}\,(iii) say that
\[
2 \nmid x; \ \ \ 
n =g-2, \ g \mid (2^{x-1}-(-1)^x),
\]
respectively.
Similarly to case where $n \equiv 0 \pmod{4}$, one has $g \equiv 1 \pmod{4}$, so that $n=g-2 \equiv 3 \pmod{4}$, a contradiction.

\vspace{0.2cm}{\it Case where $n \equiv 2 \pmod{4}.$} \ 
Lemma \ref{lem-modnplus2}\,(i,\,ii) says that we are in one of the following cases:
\[
n=2^s-2, \ s \ge 2; \ \ \ 
n=2.
\]
Thus, either $n \equiv 6 \pmod{8}$ or $n=2$.

\vspace{0.2cm}{\it Case where $n \equiv 3 \pmod{4}.$} \ 
Lemma \ref{lem-modnplus1}\,(i,\,ii) says that we are in one of the following cases:
\[
n=2^r-1, \ r \ge 2; \ \ \ 
n=3.
\]
Thus, either $n \equiv 7 \pmod{8}$ or $n=3$.

\vspace{0.2cm}
To sum up, we are in one of the following cases:
\[
n \equiv 0 \pmod{8}; \ \ n \equiv 6 \pmod{8}; \ \ n=2; \ \ n \equiv 7 \pmod{8}; \ \ n=3.
\]
Therefore, it remains to show the following lemma:

\begin{lem}\label{lem-eq(1)(-1)(-1)-n2n3}
The following hold.
\begin{itemize}
\item[$\bullet$] 
All solutions to the equation
\[ 
2^x+3^y=4^z+5^w 
\]
in positive integers $x,y,z$ and $w$ are given by $(x,y,z,w)=(1,3,1,2),$\\$(3,4,3,2),(5,2,2,2).$
\item[$\bullet$] 
The only solution to the equation
\[ 
3^x+4^y=5^z+6^w 
\]
in positive integers $x,y,z$ and $w$ is given by $(x,y,z,w)=(3,1,2,1).$
\end{itemize}
\end{lem}

\begin{proof}
This can be verified using the algorithm of Bert\'ok and Hajdu \cite{BH}.
\end{proof}

We finish the proof of assertion (ii).

\subsection{Case where $(\delta_1,\delta_2,\delta_3)=(-1,1,1)$} 
\label{subsec-(-1)(1)(-1)}
We proceed similarly to Subsections \ref{subsec-(1)(1)(-1)} and \ref{subsec-(1)(-1)(-1)}.

\vspace{0.2cm}{\it Case where $n \equiv 0 \pmod{4}.$} \ 
The combination of Lemmas \ref{lem-modnplus1}\,(iii) and \ref{lem-modnplus2}\,(iii) implies that
\[
2 \mid x, \ \ n =2g-2, \ g \mid (2^{x-1}+(-1)^x).
\]
These together imply that $g \equiv 1$ or $3 \pmod{8}$ and $n=2g-2 \equiv 0 $ or $4 \pmod{16}$.

\vspace{0.2cm}{\it Case where $n \equiv 1 \pmod{4}.$} \ 
Lemmas \ref{lem-modnplus1}\,(iii) and \ref{lem-modnplus2}\,(iii) say that
\[
2 \mid x; \ \ \ 
n =g-2, \ g \mid (2^{x-1}+(-1)^x),
\]
respectively.
Since $n \equiv 1 \pmod{4}$, these together imply that $g \equiv 3 \pmod{8}$ and $n=g-2 \equiv 1 \pmod{8}$.

\vspace{0.2cm}{\it Case where $n \equiv 2 \pmod{4}.$} \ 
Lemmas \ref{lem-mod4}\,(iii), \ref{lem-modnplus1}\,(iii) and \ref{lem-modnplus2}\,(i,\,ii) say that
\begin{align*}
&\text{
either \ $2 \mid y$ and $x>1$ \ or \ $2 \nmid y$ and $x=1;$
}\\
&\text{
$2 \mid x$, \ $n=f-1, \ f \mid (2^{w-1}+1);$
}\\
&\text{
either \ 
$2 \mid y, \ n=2^s-2, \ s \ge 2$ \ or \ $2 \nmid y, \ x=1,$
}
\end{align*}
respectively.
These together imply that
\[
n=f-1, \ f \mid (2^{w-1}+1), \ n=2^s-2.
\]
Since $f=n+1 \equiv 3 \pmod{4}$, it follows from the divisibility relation on $f$ that $w$ is even, so that $f \equiv 3 \pmod{8}$.
Thus $n =f-1 \equiv 2 \pmod{8}$.
Since $n=2^s-2$, one has $s=2$, i.e., $n=2$.

\vspace{0.2cm}{\it Case where $n \equiv 3 \pmod{4}.$} \ 
Lemma \ref{lem-modnplus1}\,(i,\,ii) says that $n=2^r-1$ for some $r \in \mathbb N$ with $r \ge 2$.
Thus, either $n=3$ or $n \equiv 7 \pmod{8}$.

\vspace{0.2cm}
To sum up, we are in one of the following cases:
\begin{alignat*}{4}
&n \equiv 0 \pmod{16};  & \ \ & n \equiv 4 \pmod{16}; & \ \ & n \equiv 1 \pmod{8}; \\
&n=2; & \ \ & n=3;  & \ \ & n \equiv 7 \pmod{8}.
\end{alignat*}
On equation \eqref{eq-(-1)(1)(1)} it is known by \cite[Theorem 3.1]{AF} that there is no solution for $n=2$, and it is known by \cite[Theorem 3.13]{AF} that the only solution is $(x,y,z,w)=(1,3,2,2)$ for $n=3$.
Thus, assertion (iii) holds.

\subsection{Case where $(\delta_1,\delta_2,\delta_3)=(-1,1,-1)$} 
We proceed similarly to Subsections \ref{subsec-(1)(1)(-1)}, \ref{subsec-(1)(-1)(-1)} and \ref{subsec-(-1)(1)(-1)}.

\vspace{0.2cm}{\it Case where $n \equiv 0 \pmod{4}.$} \ 
Lemmas \ref{lem-modnplus1}\,(ii,\,iii) and \ref{lem-modnplus2}\,(iii) say that
\[
2 \mid x; \ \ \ 
n =2g-2, \ g \mid (2^{x-1}-(-1)^x),
\]
respectively.
These together imply that $n \equiv 0$ or $12 \pmod{16}$.

\vspace{0.2cm}{\it Case where $n \equiv 1 \pmod{4}.$} \ 
Lemmas \ref{lem-modnplus1}\,(ii,\,iii) and \ref{lem-modnplus2}\,(iii) say that
\[
2 \mid x; \ \ \ 
n =g-2, \ g \mid (2^{x-1}-(-1)^x),
\]
respectively.
These together imply that $n \equiv 5 \pmod{8}$.

\subsubsection{Case where $n \equiv 2 \pmod{4}$} 
Lemma \ref{lem-modnplus2}\,(i,\,ii) says that $n=2^s-2$ for some $s \ge 2$.
Thus, either $n=2$ or $n \equiv 6 \pmod{8}$.

\subsubsection{Case where $n \equiv 3 \pmod{4}$} 
Lemmas \ref{lem-modnplus1}\,(iii) says that we are in one of the following cases:
\[
2 \nmid x, \ n=2^r-1, \ r \ge 2; \ \ \ 
2 \mid x. 
\]
Also, Lemma \ref{lem-modnplus2}\,(iii) says that
\[
n=g-2, \ g \mid (2^{x-1}-(-1)^x).
\]
If $x$ is odd, then either $n=3$ or $n=2^r-1 \equiv 7 \pmod{8}$. 
Also, since $g=n+2 \equiv 1 \pmod{4}$, one finds that if $x$ is even, then the divisibility relation on $g$ implies that $g \equiv 1 \pmod{8}$, so that $n=g-2 \equiv 7 \pmod{8}$.

\vspace{0.2cm}
To sum up, we are in one of the following cases:
\begin{alignat*}{4}
&n \equiv 0 \pmod{16}; & \ \ &  
n \equiv 12 \pmod{16}; & \ \ &  
n \equiv 5 \pmod{8}; & \ \ & 
n=2;\\ 
&n \equiv 6 \pmod{8}; & \ \ &  
n=3; & \ \ &  
n \equiv 7 \pmod{8}.
\end{alignat*}
Therefore, it remains to show the following lemma:

\begin{lem}\label{lem-eq(-1)(1)(-1)-n2n3}
The following hold.
\begin{itemize}
\item[$\bullet$] 
The only solution to the equation
\[ 
2^x+4^z=3^y+5^w 
\]
in positive integers $x,y,z$ and $w$ are given by $(x,y,z,w)=(2,1,1,1),$\\$(4,3,2,1),(6,1,3,3).$
\item[$\bullet$] 
The only solution to the equation
\[ 
3^x+5^z=4^y+6^w 
\]
in positive integers $x,y,z$ and $w$ is given by $(x,y,z,w)=(3,2,2,2).$
\end{itemize}
\end{lem}

\begin{proof}
This can be verified using the algorithm of Bert\'ok and Hajdu \cite{BH}.
\end{proof}

We finish the proof of assertion (iv), and complete the proof of Theorem \ref{th-miscellaneous}.

\subsection*{Acknowledgements}
The authors are indebted to Lajos Hajdu for his kind help to use the algorithm of \cite{Ha} in the proofs of Lemmas \ref{lem-eq(1)(-1)(-1)-n2n3} and \ref{lem-eq(-1)(1)(-1)-n2n3}.

\vskip.2cm

\noindent{\bf Data availability.} 
Data sets generated during the current study are available from the corresponding author on reasonable request.

\end{document}